\documentclass{article}   
\usepackage[latin1]{inputenc}
\usepackage[T1]{fontenc}
\usepackage[english]{babel}
\usepackage{amsmath}   
\usepackage{amsfonts}
\usepackage{amssymb}
\usepackage{amsthm}
\usepackage{color}
\usepackage{multirow}
\usepackage{hyperref}

\newcommand{\md}{\mathrm{d}}
\newcommand{\lag}{\langle}
\newcommand{\rag}{\rangle}

\renewcommand{\AA}{\mathcal{A}}

\newcommand{\FF}{\mathcal{F}}

\newcommand{\PP}{\mathcal{P}}

\newcommand{\R}{\mathbb{R}}

\newcommand{\N}{\mathbb{N}}
\newcommand{\Z}{\mathbb{Z}}
\newcommand{\T}{\mathbb{T}}

\numberwithin{equation}{section}

\theoremstyle{definition}
\newtheorem{defi}{Definition}[section]

\theoremstyle{plain}
\newtheorem{prop}{Proposition}[section]
\newtheorem{theo}{Theorem}[section]
\newtheorem{lemme}{Lemma}[section]

\theoremstyle{remark}
\newtheorem{rmq}{Remark}[section]

\title{Approximate controllability for a \textsc{2d} Grushin equation with potential having an internal singularity}
\author{
Morgan \textsc{Morancey}
\footnote{I2M UMR 7373, Universit\'e Aix-Marseille.   
email: morgan.morancey@univ-amu.fr}
\thanks{The author was partially supported by the ``Agence Nationale de la Recherche'' (ANR),
Projet Blanc EMAQS number ANR-2011-BS01-017-01 and by CMLA UMR 8536, ENS Cachan, 61 avenue du Président Wilson, 94235 Cachan, FRANCE.} }
\date{}

\begin{document}

\maketitle

\hrule
\begin{abstract}
This paper is dedicated to approximate controllability for Grushin equation on the rectangle $(x,y) \in (-1,1) \times (0,1)$ with an inverse square potential. This model corresponds to the heat equation for the Laplace-Beltrami operator associated to the Grushin metric on $\R^2$, studied by Boscain and Laurent. The operator is both degenerate and singular on the line $\{ x=0 \}$. 

The approximate controllability is studied through unique continuation of the adjoint system.
For the range of singularity under study, approximate controllability is proved to hold whatever the degeneracy is.

Due to the internal inverse square singularity, a key point in this work is the study of well-posedness. An extension of the singular operator is designed imposing suitable transmission conditions through the singularity.

Then, unique continuation relies on the Fourier decomposition of the \textsc{2d} solution in one variable and Carleman estimates for the \textsc{1d} heat equation solved by the Fourier components. The Carleman estimate uses a suitable Hardy inequality.
\end{abstract}
\hrule


\section{Introduction}

\subsection{Main result}

\noindent
We consider for $\gamma >0$ the following degenerate singular parabolic equation
\begin{equation} \label{grushin_2d}
\left\{
\begin{aligned}
& \partial_t f - \partial^2_{xx} f - |x|^{2 \gamma} \partial^2_{yy} f + \frac{c}{x^2} f = u(t,x,y) \chi_{\omega}(x,y),   &(t,x,y) \in (0,T) \times \Omega,&
\\
&f(t,x,y) =  0, & (t,x,y) \in (0,T) \times \partial \Omega,&
\end{aligned}
\right.
\end{equation}
with initial condition
\begin{equation} \label{grushin_2d_ci}
f(0,x,y) = f^0(x,y), \quad (x,y) \in \Omega.
\end{equation}
The domain is $\Omega := (-1,1) \times (0,1)$ and $\omega$, the control domain, is an open subset of $\Omega$. The function $\chi$ is the indicator function. The coefficient $c$ of the singular potential is real and will be restricted to $\left( -\frac{1}{4}, \frac{3}{4} \right)$. The degeneracy set $\{x=0\}$ coincides with the singularity set ; it separates the domain $\Omega$ in two connected components.
Due to the singular potential, the first difficulty is to give a meaning to solutions of (\ref{grushin_2d}). Through the study of an associated \textsc{1d} heat equation, we will design a suitable extension of the considered operator generating a continuous semigroup. The solutions considered in this article will be related to this semigroup. This is detailed in Section~\ref{sect_bien_pose}. Before stating the controllability result, we give some motivations and justify the range of constants $c \in \left( -\frac{1}{4}, \frac{3}{4} \right)$.

\medskip
In \cite{BoscainLaurent}, Boscain and Laurent studied the Laplace-Beltrami operator for the Grushin-like metric given by the orthonormal basis $X= \begin{pmatrix} 1 \\ 0 \end{pmatrix}$ and $Y = \begin{pmatrix} 0 \\ |x|^{\gamma} \end{pmatrix}$ on $\R \times \mathbb{T}$ with $\gamma >0$ i.e.
\begin{equation} \label{laplace-beltrami}
L u := \partial^2_{xx} u + |x|^{2 \gamma} \partial^2_{yy} u - \frac{\gamma}{x} \partial_x u.
\end{equation}
They proved that this operator with domain $C^\infty_0\big( (\R \backslash \{0\}) \times \mathbb{T} \big)$ is essentially self-adjoint on $L^2(\R\times \mathbb{T})$ if and only if $\gamma \geq 1$. Thus, for the heat equation associated to this Laplace-Beltrami operator, no information passes through the singular set $\{x=0\}$ when $\gamma \geq 1$. This prevents controllability from one side of the singularity.

\smallskip
The change of variables $u= |x|^{\gamma/2} v$, leads to study
\begin{equation} \label{laplace-beltrami_chgtvar}
\Delta_\gamma v = \partial^2_{xx} v + |x|^{2 \gamma} \partial^2_{yy} v - \frac{\gamma}{2} \left( \frac{\gamma}{2} +1 \right) \frac{v}{x^2}.
\end{equation}
The model (\ref{grushin_2d}) can then be seen as a heat equation for this operator. By choosing a coefficient $c$ instead of $\frac{\gamma}{2} \left( \frac{\gamma}{2} +1 \right)$ we authorize a wider class of singular potentials and decouple the effects of the degeneracy and the singularity for a better understanding of each one of these phenomena. Adapting the arguments of \cite{BoscainLaurent}, one obtains that for any $\gamma >0$, the operator $-\partial^2_{xx} -|x|^{2\gamma} \partial^2_{yy} + \frac{c}{x^2}$ with domain $C^\infty_0(\Omega \backslash \{x=0\})$ is essentially self-adjoint on $L^2(\Omega)$ if and only if $c \geq \frac{3}{4}$.
Thus, to look for controllability properties, our study focuses on the range of constants $c < \frac{3}{4}$.

\smallskip
The lower bound $c > -\frac{1}{4}$ for the range of constants considered comes from well posedness issues linked to the use of the following Hardy inequality (see e.g. \cite{CannarsaMartinezVancostenoble08} for a simple proof)
\begin{equation} \label{Hardy_1d_(0,1)}
\int_0^1 \frac{z(x)^2}{x^2} \md x \leq 4 \int_0^1 z_x(x)^2 \md x, \quad \forall z \in H^1((0,1),\R) \text{ with } z(0)=0.
\end{equation}
The critical case in the Hardy inequality $c= -\frac{1}{4}$ is not directly covered by the technics of this article.

\medskip
\noindent
The notion of controllability under study in this article is given in the following definition.
\begin{defi} \label{Def:ApproxControl}
Let $T>0$ and $\omega \subset \Omega$. The problem (\ref{grushin_2d}) is said to be approximately controllable from $\omega$ in time $T$ if for any $(f^0 , f^T) \in L^2(\Omega)^2$, for any $\varepsilon >0$, there exists $u \in L^2((0,T) \times \omega)$ such that the solution of (\ref{grushin_2d})-(\ref{grushin_2d_ci}), in the sense of Proposition~\ref{prop_bien_pose_2d}, satisfies
\begin{equation*}
|| f(T) - f^T ||_{L^2(\Omega)} \leq \varepsilon.
\end{equation*}
\end{defi}

\noindent
The main result of this article is the following characterization of approximate controllability.
\begin{theo} \label{th_controle_approche_2d}
Let $T>0$, $\gamma >0$ and $c \in (-\frac{1}{4}, \frac{3}{4})$. Let $\omega$ be an open subset of $\Omega$. Then, (\ref{grushin_2d}) is approximately controllable from $\omega$ in time $T$ in the sense of Definition~\ref{Def:ApproxControl}.
\end{theo}
Except for the critical case of the Hardy inequality ($c=-\frac{1}{4}$), this theorem fills the gap, for the approximate controllability property, between validity of Hardy inequality ($c \geq -\frac{1}{4}$) and the essential self-adjointness property of \cite{BoscainLaurent} for $c \geq \frac{3}{4}$.

\begin{rmq}
One key point for this approximate controllability result is to give a meaning to the solutions of (\ref{grushin_2d}). As it will be noticed (see e.g. Sect.~\ref{subsect_extensions}) there are various possible definitions of solutions (depending mostly on what transmission conditions are imposed at the singularity). The validity of the approximate controllability property under study will strongly depend on these transmission conditions. This is why, in Definition~\ref{Def:ApproxControl}, it is precised that the solutions are understood in the sense of Proposition~\ref{prop_bien_pose_2d}.
\end{rmq}

\begin{rmq}
Going back to the Laplace-Beltrami operator studied by Boscain and Laurent (\ref{laplace-beltrami_chgtvar}), we would get approximate controllability for the heat equation associated to the operator $\Delta_\gamma$ for any $\gamma \in (0,1)$. 
To be closer to the setting they studied one can notice that, essentially with the same proof, the approximate controllability result of Theorem~\ref{th_controle_approche_2d} also holds on $(-1,1) \times \T$. This will be detailed in Remark~\ref{rk_conditions_periodiques}.
\end{rmq}

By a classical duality argument, approximate controllability will be studied through unique continuation of the adjoint system.
The unique continuation result will be proved by a suitable Carleman inequality for an associated sequence of \textsc{1d} problems. This Carleman estimate rely on a precise Hardy inequality.

\bigskip
The model (\ref{grushin_2d}) can also be seen as an extension of \cite{BeauchardCannarsaGuglielmi} where Beauchard \textit{et al.} studied the null controllability without the singular potential (i.e. in the case $c=0$).
The authors proved that null controllability holds if $\gamma \in (0,1)$ and does not hold if $\gamma >1$. In the case $\gamma=1$, for $\omega$ a strip in the $y$ direction, null controllability holds if and only if the time is  large enough.

The inverse square potential for the Grushin equation has already been taken into account by Cannarsa and Guglielmi in \cite{CannarsaGuglielmi13} but in the case where both degeneracy and singularity are at the boundary. With our notations, they proved null controllability in sufficiently large time for $\Omega = (0,1) \times (0,1)$, $\omega = (a,b) \times (0,1)$, $\gamma =1$ and any $c > - \frac{1}{4}$. They also proved that approximate controllability holds for any control domain $\omega \subset \Omega$, any $\gamma >0$ and any $c > -\frac{1}{4}$. Thus, the fact that our model presents an internal singularity instead of a boundary singularity deeply affects the approximate controllability property as it does not hold when $c> \frac{3}{4}$.

\noindent
As in \cite{BeauchardCannarsaGuglielmi}, the results of this article will strongly use an associated sequence of \textsc{1d} problems. 
As a by-product of the proof of Theorem~\ref{th_controle_approche_2d}, we obtain the following approximate controllability result for the \textsc{1d} heat equation with a singular inverse square potential.
\begin{theo} \label{th_controle_approche_chaleur_1d}
Let $T > 0$ and $c \in \left( -\frac{1}{4}, \frac{3}{4} \right)$. Let $\omega$ be an open subset of $(-1,1)$. Then approximate controllability holds for 
\begin{equation} \label{chaleur_1d}
\left\{
\begin{aligned}
&\partial_t f - \partial^2_{xx} f + \frac{c}{x^2} f = u(t,x) \chi_{\omega}(x), &(t,x)& \in (0,T) \times (-1,1),
\\
&f(t,-1) = f(t,1) = 0, &t& \in (0,T),
\\
&f(0,x) = f^0(x),  &x& \in (-1,1),
\end{aligned}
\right.
\end{equation}
where the solutions of (\ref{chaleur_1d}) are given by Proposition~\ref{prop_bien_pose_1d}.
\end{theo}
The null controllability issue for the \textsc{1d} heat equation with such an internal inverse square singularity remains an open question. Like (\ref{grushin_2d}), it has to be noticed that the solutions of (\ref{chaleur_1d}) are related to the semigroup generated by a suitable extension of the Laplace operator with a singular potential.

\subsection{Structure of this article}

Due to the internal singularity and the fact that the considered operators admit several self-adjoint extensions, the functional setting and the well posedness are crucial issues in this article. Section \ref{sect_bien_pose} is dedicated to these questions.

Section~\ref{sect_continuation_unique} is dedicated to the study of the unique continuation property for the adjoint system. Using decomposition in Fourier series in the $y$ variable and unique continuation for uniformly parabolic operator we reduce the problem to the study of a \textsc{1d} singular problem with a boundary inverse square potential. Then we conclude proving a suitable Carleman inequality using an adapted Hardy's inequality.

We end this introduction by a brief review of previous results concerning degenerate and/or singular parabolic equations.

\subsection{A review of previous results}

\noindent
The first result for a heat equation with an inverse square potential $\frac{c}{\|x\|^2}$ deals with well posedness issues. In \cite{BarasGoldstein84}, Baras and Goldstein proved complete instantaneous blow-up for positive initial conditions in space dimension $N$ if $c < c^*(N) :=-\frac{(N-2)^2}{4}$. This critical value is the best constant in Hardy's inequality. Cabr\'e and Martel also studied in \cite{CabreMartel99}  the relation between blow-up of such equations and the existence of an Hardy inequality.
Thus, most of the following studies focus on the range of constants $c \geq c^*(N)$. In this case, well posedness in $L^2(\Omega)$ has been proved in \cite{VazquezZuazua00} by Vazquez and Zuazua. Notice that in those cases the singular set is the point $\{ 0 \}$ (the singularity being at the boundary in the one dimensional case) whereas in this article the singular set is a line separating the \textsc{2d} domain in two connected components.

\noindent
The controllability issues were first studied for degenerate equations. In~\cite{CannarsaMartinezVancostenoble05,MartinezVancostenoble06, CannarsaMartinezVancostenoble08, CannarsaMartinezVancostenoble09}, Cannarsa, Martinez and Vancostenoble proved null controllability with a distributed control for a one dimensional parabolic equation degenerating at the boundary. 
Then, they extended this result to more general degeneracies and in dimension two. These results are based on suitable Hardy inequalities and Carleman estimates. 
More recently, Cannarsa, Tort and Yamamoto~\cite{CannarsaTortYamamoto12} proved approximate controllability for this one dimensional equation degenerating at the boundary with a Dirichlet control on the degenerate boundary. 
Then, Gueye~\cite{Gueye13} proved null controllability for  the same model. Its proof relies on transmutation and appropriate nonharmonic Fourier series.

\noindent
Meanwhile, these Carleman estimates were adapted for heat equation with an inverse square potential $\frac{c}{\|x\|^2}$ in dimension $N \geq 3$. In \cite{VancostenobleZuazua}, Vancostenoble and Zuazua proved null controllability in the case where the control domain $\omega$ contains an annulus centred on the singularity. Their proof relies on a decomposition in spherical harmonics reducing the problem to the study of a \textsc{1d} heat equation with an inverse square potential which is singular at the boundary. The geometric assumptions on the control domain were then removed by Ervedoza in \cite{Ervedoza08} using a direct Carleman strategy in dimension $N \geq 3$. Notice that although these results deal with internal singularity they cannot be adapted to our setting. Indeed, in \cite{VancostenobleZuazua} it is crucial that the singularity of the \textsc{1d} problem obtained by decomposition in spherical harmonics is at the boundary. The Carleman strategy developed in \cite{Ervedoza08} cannot be adapted in this article because our singularity is no longer a point but separates the domain in two connected components.

\noindent 
For null controllability for a one dimensional parabolic equation both degenerate and singular at the boundary we refer to \cite{Vancostenoble11} by Vancostenoble. The proof extends the previous Carleman strategy together with an improved Hardy inequality.

\noindent
As the functional setting for this study is obtained through the design of a suitable self-adjoint extension of our Grushin-like operator, we mention the work \cite{BoscainPrandi13} conducted simultaneously to this study. In this paper, Boscain and Prandi studied some extensions of the Laplace-Beltrami operator (\ref{laplace-beltrami}) for $\gamma \in \R$. Among other things, they designed for a suitable range of constants an extension called bridging extension that allows full communication through the singular set. Even if the models under consideration are not exactly the same, the approximate controllability from one side of the singularity given by Theorem~\ref{th_controle_approche_2d} is in agreement with the existence of this bridging extension.

\section{Well posedness}
\label{sect_bien_pose}

The previous results of the literature dealing with an inverse square potential were obtained thanks to some Hardy-type inequality. 
For a boundary inverse square singularity (as in \cite{Vancostenoble11}), the condition $z(0)=0$ needed for (\ref{Hardy_1d_(0,1)}) to hold is contained in the homogeneous Dirichlet boundary conditions considered. Thus, in \cite{Vancostenoble11}, the appropriate functional setting to study the \textsc{1d} operator $-\partial^2_{xx} + \frac{c}{x^2}$ with $c > -\frac{1}{4}$ is
\begin{equation*}
\left\{ f \in H^2_{loc}((0,1]) \cap H^1_0(0,1) \, ; \, -\partial^2_{xx} f  + \frac{c}{x^2} f \in L^2(0,1) \right\}.
\end{equation*}
For an internal inverse square singularity one still has
\begin{equation} \label{Hardy_1d_(-1,1)}
\int_{-1}^1 \frac{z(x)^2}{x^2} \md x \leq 4 \int_{-1}^1 z_x(x)^2 \md x, \quad \forall z \in H^1(-1,1) \text{ such that } z(0)=0.
\end{equation}
This inequality ceases to be true if $z(0) \neq 0$. Thus, the functional setting must contain some informations on the behaviour of the functions at the singularity.

In this section, we design a suitable self-adjoint extension of the operator $-\partial^2_{xx} -|x|^{2 \gamma} \partial^2_{yy} + \frac{c}{x^2}$ on $C^\infty_0(\Omega\backslash \{x=0\})$. 
The next subsection deals with an associated one dimensional equation. Section~\ref{subsect_bien_pose_2d} will then relate this one dimensional problem to the original problem in dimension two.
In all what follows, the coefficient of the singular potential will be parametrized in the form $c = c_\nu$ where
\begin{equation} \label{def_c_nu}
c_\nu := \nu^2 - \frac{1}{4}, \quad \text{for } \nu \in (0,1).
\end{equation}

\subsection{Introduction of the \textsc{1d} operator}

For $n \in \N^*$, $\gamma > 0$ and $\nu \in (0,1)$ we consider the following homogeneous problem
\begin{equation} \label{grushin_1d}
\left\{
\begin{aligned}
& \partial_t f - \partial^2_{xx} f + \frac{c_\nu}{x^2} f + (n \pi)^2 |x|^{2 \gamma} f = 0,  &(t,x)& \in (0,T) \times (-1,1),
\\
& f(t,-1) = f(t,1) = 0,    &t& \in (0,T).
\end{aligned}
\right.
\end{equation}
This equation is formally the homogeneous equation satisfied by the coefficients of the Fourier expansion in the $y$ variable done in \cite{BeauchardCannarsaGuglielmi} and will be linked to (\ref{grushin_2d}) in Sect.~\ref{subsect_bien_pose_2d}. From now on, we focus on the well posedness of (\ref{grushin_1d}).

\begin{rmq}
A naive functional setting for this equation is the adaptation of~\cite{Vancostenoble11}
\begin{align*}
\Big\{&  f \in L^2(-1,1) \, ; \,f_{|[0,1]} \in H^2_{loc}((0,1]) \cap H^1_0(0,1), 
\\ 
&f_{|[-1,0]} \in H^2_{loc}([-1,0)) \cap H^1_0(-1,0) \text{ and }  -\partial^2_{xx} f + \frac{c_\nu}{x^2} f \in L^2(-1,1) \Big\}.
\end{align*}
However, a functional setting where the two problems on $(-1,0)$ and $(0,1)$ are well posed is not pertinent for the control problem from one side of the singularity. It leads to decoupled dynamics on the connected components of $(-1,0) \cup (0,1)$.
\end{rmq}

\smallskip
We study the differential operator 
\begin{equation*}
A_n f(x) := - \partial^2_{xx} f(x) + \frac{c_\nu}{x^2} f(x) + (n \pi)^2 |x|^{2 \gamma} f(x).
\end{equation*} 
As $\nu \in (0,1)$, the results of \cite{BoscainLaurent} imply that $A_n$ defined on $C_0^\infty ( (-1,0) \cup (0,1))$ admits several self-adjoint extensions. We here specify the self-adjoint extension that will be used. Let
\begin{equation*}
\tilde{H}^2_0(-1,1) := \left\{ f \in H^2(-1,1) \, ; \, f(0)=f'(0)=0 \right\},
\end{equation*}
and
\begin{align*}
\FF_s := \Big\{ f \in L^2(-1,1) \, ; \, &f= c_1^+ |x|^{\nu+\frac{1}{2}} + c_2^+ |x|^{-\nu+\frac{1}{2}} \text{ on } (0,1) 
\\
\text{ and } &f=c_1^- |x|^{\nu+\frac{1}{2}} + c_2^- |x|^{-\nu+\frac{1}{2}} \text{ on } (-1,0) \Big\}.
\end{align*}
Notice that for any  $f_s \in \FF_s$,
\begin{equation} \label{pb_sturm_fs}
\left(-\partial^2_{xx}  + \frac{c_\nu}{x^2} \right) f_s(x) = 0, \quad \forall x \in (-1,0) \cup (0,1).
\end{equation}
The parametrization (\ref{def_c_nu}) of the coefficient of the singular potential by $\nu$ allows to write easily the functions of $\FF_s$.

\medskip
The domain of the operator is defined by
\begin{align*} 
D(A_n) := \Big\{& f=f_r + f_s  \,; \, f_r \in \tilde{H}^2_0(-1,1), \, f_s \in \FF_s \text{ such that } f(-1)=f(1)=0, 
\\
& c_1^- + c_2^- + c_1^+ + c_2^+ = 0 
\text{ and }
\\
&(\nu+\frac{1}{2}) c_1^- + (-\nu+\frac{1}{2}) c_2^- = (\nu+\frac{1}{2}) c_1^+ + (-\nu+\frac{1}{2}) c_2^+ \Big\}.
\label{def_D(A)}
\tag{\theequation} \addtocounter{equation}{1}
\end{align*}
Notice that for $\nu \in (0,1)$, $D(A_n) \subset L^2(-1,1)$. In the following, this unique decomposition of functions of $D(A_n)$ will be referred to as the regular part for $f_r$ and the singular part for $f_s$.
As this domain is independent of $n$, it will be denoted by $D(A)$ in the rest of this article.
The coefficients of the singular part will be denoted by $c_1^+$ if there is no ambiguity and $c_1^+(f)$ otherwise. The conditions imposed on these coefficients in (\ref{def_D(A)}) will be referred to as the \emph{transmission conditions}. These conditions are discussed in Remark~\ref{rk_cond_transmission}, their role and origin are discussed in Sect.~\ref{subsect_bien_pose_1d} and~\ref{subsect_extensions}.

This operator satisfies the following properties
\begin{prop} \label{prop_auto_adjoint}
For any $n \in \N^*$ and $\nu \in (0,1)$, the operator $(A_n,D(A))$ is self-adjoint on $L^2(-1,1)$. Moreover, for any $f \in D(A)$,
\begin{equation} \label{coercivite_1d}
\lag A_n f , f \rag \geq m_\nu \int_{-1}^1 \partial_x f_r (x)^2 \md x + (n \pi)^2 \int_{-1}^1 |x|^{2 \gamma} f(x)^2 \md x,
\end{equation}
where $m_\nu := \min \{ 1, 4 \nu^2 \}$.
\end{prop}
Before proving this proposition in Sect.~\ref{subsect_bien_pose_1d}, we give some comments on this construction of the $1$\textsc{d} operator.

\begin{rmq}
As noticed in (\ref{pb_sturm_fs}), the functions of $\FF_s$ are chosen in the kernel of the singular differential operator $-\partial^2_{xx} + \frac{c_\nu}{x^2}$. Thus, for any $f \in D(A)$,
\begin{equation*}
A_n f = \left(-\partial^2_{xx} f_r + \frac{c_\nu}{x^2} f_r \right) + (n \pi)^2 |x|^{2 \gamma} f.
\end{equation*}
As done in~\cite[Proposition 3.1]{AlekseevaAnanieva}, for any $f_r \in \tilde{H}^2_0(-1,1)$, writing
\begin{equation*}
f_r(x) = \int_0^x (x-s) f_r''(s) \md s,
\end{equation*}
and applying Minkowski's integral inequality we get that the map $x \mapsto \frac{1}{x^2} f_r(x)$ belongs to  $L^2(-1,1)$. Thus, $(A_n, D(A))$ is indeed an operator in $L^2(-1,1)$.
\end{rmq}

\begin{rmq} \label{rk_cond_transmission}
The reason for imposing these particular transmission conditions is threefold. First, it implies the self-adjointness of the operator under consideration. This will be pointed out in the proof of Proposition~\ref{prop_auto_adjoint}. 
This choice is guided by the general theory of self-adjoint extensions from the point of view of boundary conditions as detailed by Zettl~\cite[Theorem 13.3.1, Case 5]{ZettlBook}. For the sake of clarity, the proof of self-adjointness is done independently of this general theory in Sect.~\ref{subsect_bien_pose_1d}. A discussion relating this general theory and the domain (\ref{def_D(A)}) together with other possible choices is done in Sect.~\ref{subsect_extensions}.

The second interest of these transmission conditions is to ensure the positivity of the operator, as detailed in the proof of Proposition~\ref{prop_bien_pose_1d}.

Finally, these transmission conditions are really \emph{transmission} conditions in the sense that they allow some information to cross the singularity. In matrix form, the transmission conditions can be rewritten as
\begin{equation}  \label{cond_transmission_matricielle}
\begin{pmatrix} c_1^+(f) \\ c_2^+(f) \end{pmatrix} =
\frac{-1}{2 \nu} \begin{pmatrix} -1 & 2 \nu -1   \\  2\nu+1 & 1  \end{pmatrix}
\begin{pmatrix} c_1^-(f) \\ c_2^-(f) \end{pmatrix}, \quad \forall f \in D(A).
\end{equation}
Thus, the invertibility of the above matrix implies that if the singular part of some function $f \in D(A)$  identically vanishes on one side of the singularity it also vanishes on the other side. This is a crucial point for the proof of approximate controllability.
\end{rmq}

\medskip
\noindent
Using Proposition \ref{prop_auto_adjoint}, the well posedness of the one dimensional system (\ref{grushin_1d}) follows from Proposition~\ref{prop_auto_adjoint} and the Hille-Yosida theorem (see e.g. \cite[Theorem 3.2.1]{CazenaveHarauxBook}).
\begin{prop} \label{prop_bien_pose_1d}
For any $n \in \N^*$ and any $f^0 \in L^2(-1,1)$, problem (\ref{grushin_1d}) with initial condition $f(0,\cdot)=f^0$  has a unique solution 
\begin{equation*}
f \in C^0([0,+\infty), L^2(-1,1)) \cap C^0((0,+\infty),D(A)) \cap C^1((0,+\infty),L^2(-1,1)).
\end{equation*}
This solution satisfies
\begin{equation*}
||f(t)||_{L^2(-1,1)} \leq ||f^0||_{L^2(-1,1)}.
\end{equation*}
\end{prop}

\noindent
In all what follows, we denote by $e^{-A_n t}$ the semigroup generated by $-A_n$ i.e. for any $f^0 \in L^2(-1,1)$, the function $t \mapsto e^{-A_n t } f^0$ is the solution of (\ref{grushin_1d}) given by Proposition~\ref{prop_bien_pose_1d}.

\subsection{Well posedness of the \textsc{1d} problem}
\label{subsect_bien_pose_1d}

This subsection is dedicated to the proof of Proposition~\ref{prop_auto_adjoint}. The proof uses the following two lemmas.

The following lemma is proved in~\cite[Lemma 9.2.3]{ZettlBook}.
\begin{lemme} \label{lemme_forme_Lagrange}
For $f,g \in \tilde{H}^2_0(-1,1) \oplus \FF_s$, if we define
\begin{equation*}
[f,g] (x) := (f g' - f' g)(x), \quad \forall x \neq 0,
\end{equation*}
then
\begin{align*}
\int_{-1}^1 \left(-\partial^2_{xx} f + \frac{c_\nu}{x^2} f \right)(x) g(x) \md x 
&= \int_{-1}^1 f(x) \left(-\partial^2_{xx} g + \frac{c_\nu}{x^2} g \right)(x) \md x 
\\
&+ [f,g](1) - [f,g](0^+) + [f,g](0^-) - [f,g](-1).
\end{align*}
\end{lemme}

\noindent
The following lemma characterizes the behaviour of the regular part at the singularity.
\begin{lemme} \label{lemme_fr_en_0}
For any $f \in H^2(0,1)$, satisfying $f(0)=f'(0)=0$,
\begin{equation*}
\lim\limits_{x \to 0} \frac{f(x)}{x^{\frac{3}{2}}} = 0 
\quad \text{ and } \quad
\lim\limits_{x \to 0} \frac{f'(x)}{x^{\frac{1}{2}}} = 0.
\end{equation*}
The same holds for functions in $\tilde{H}^2_0(-1,1)$ and both limits $x \to 0^{\pm}$.
\end{lemme}

\begin{proof}[Proof of Lemma \ref{lemme_fr_en_0}]
As $f(0) = f'(0) =0$, it comes that
\begin{equation*}
f(x) = \int_0^x \int_0^t f''(s) \md s \md t.
\end{equation*}
Then, Cauchy-Schwarz inequality implies,
\begin{equation*}
|f(x)| \leq  \int_0^x  \sqrt{t} \left( \int_0^t |f''(s)|^2 \md s \right)^{\frac{1}{2}} \md t
\leq \frac{2}{3} \left(\int_0^x |f''(s)|^2 \md s \right)^{\frac{1}{2}} x^{\frac{3}{2}}.
\end{equation*}
The proof of the second limit is similar.

\end{proof}

\noindent
We now turn to the proof of Proposition \ref{prop_auto_adjoint}.
\smallskip

\begin{proof}[Proof of Proposition \ref{prop_auto_adjoint}] 
We start by proving that $(A_n,D(A))$ is a symmetric operator. Thus, $A_n^*$ is an extension of $A_n$ and self-adjointness will follow from the equality $D(A_n^*)=D(A_n)$.

\medskip
\textit{First step :  we prove that $(A_n,D(A))$ is a symmetric operator.}

\noindent
Let $f,g \in D(A)$. As $f(1) =g(1) = f(-1) =g(-1) = 0$, it comes that
\begin{equation*}
[f,g](1) = [f,g](-1) = 0.
\end{equation*}
Lemma \ref{lemme_fr_en_0} implies that
\begin{equation*}
[f,g](0^+) = [f_s,g_s](0^+) = \big( c_1^+(f) c_2^+(g) - c_2^+(f) c_1^+(g) \big) [|x|^{\nu+\frac{1}{2}}, |x|^{-\nu+\frac{1}{2}}](0^+),
\end{equation*}
and
\begin{align*}
[f,g](0^-) &= [f_s,g_s](0^-) 
\\
&= \big( c_1^-(f) c_2^-(g) - c_2^-(f) c_1^-(g) \big) [|x|^{\nu+\frac{1}{2}}, |x|^{-\nu+\frac{1}{2}}](0^-)
\\
&= - \big( c_1^-(f) c_2^-(g) - c_2^-(f) c_1^-(g) \big) [|x|^{\nu+\frac{1}{2}}, |x|^{-\nu+\frac{1}{2}}](0^+).
\end{align*}
Thus, using the matrix formulation~(\ref{cond_transmission_matricielle}) of the transmission conditions, we get that for any $f,g \in D(A)$
\begin{equation*}
c_1^+(f) c_2^+(g) - c_2^+(f) c_1^+(g) = 
- \big( c_1^-(f) c_2^-(g) - c_2^-(f) c_1^-(g) \big).
\end{equation*}
This leads to
\begin{equation*}
[f,g](0^+) = [f,g](0^-).
\end{equation*}
Finally, Lemma~\ref{lemme_forme_Lagrange} implies that for any $f,g \in D(A)$, $\lag A_n f ,g \rag = \lag f, A_n g \rag$.
\smallskip

Thus, to prove self-adjointness it remains to prove that $D(A_n^*)=D(A)$. As $D(A)$ is independent of $n$ and $x \mapsto (n \pi)^2 |x|^{2\gamma} \in L^{\infty}(-1,1)$ it comes that $D(A_n^*)=D(A_0^*)$.

\medskip
\textit{Second step : minimal and maximal domains.} First, we explicit the minimal and maximal domains in the case of a boundary singularity. Without loss of generality, we study the operator on $(0,1)$.

\noindent
Using \cite[Proposition 3.1]{AlekseevaAnanieva}, the minimal and maximal domains associated to the differential expression $A_0$ in $L^2(0,1)$ are respectively equal to
\begin{equation*}
H^2_0([0,1]) := \left\{ y \in H^2([0,1]) \, ; \, y(0)=y(1)=y'(0)=y'(1)=0 \right\}
\end{equation*}
and
\begin{equation*}
\left\{ y \in H^2([0,1]) \, ; \, y(0)=y'(0)=0 \right\} \oplus \text{Span} \left\{ x^{\nu+\frac{1}{2}}, \, x^{-\nu+\frac{1}{2}} \right\}.
\end{equation*}

\noindent
Then, \cite[Lemma 13.3.1]{ZettlBook} imply that the minimal and maximal domains associated to $A_0$ on the interval $(-1,1)$ are given by 
\begin{equation} \label{def_Dmin}
D_{min} := \left\{ f \in \tilde{H}^2_0(-1,1) \, ; \, f(-1)=f(1)=f'(-1)=f'(1)=0 \right\},
\end{equation}
and
\begin{equation} \label{def_Dmax}
D_{max} := \tilde{H}^2_0(-1,1) \oplus \FF_s.
\end{equation}
Besides, the minimal and maximal operators form an adjoint pair.

\medskip
\textit{Third step : self-adjointness.}
The operator $A_0$ being a symmetric extension of the minimal operator it comes that $D(A_0) \subset D(A_0^*) \subset D_{max}$. Let $g \in D(A_0^*)$ be decomposed as $g = g_r + g_s$ with $g_r \in \tilde{H}^2_0(-1,1)$ and $g_s \in \FF_s$. We prove that $g$ satisfy the boundary and transmission conditions. By the definition of $D(A_0^*)$, there exists $c>0$ such that for any $f \in D(A)$,
\begin{equation*}
|\lag A_0 f , g \rag| \leq c ||f||_{L^2}.
\end{equation*}
Let $f \in D(A) \cap \tilde{H}^2_0(-1,1)$ be such that $f \equiv 0$ in $(-1,0)$. Then, Lemma \ref{lemme_forme_Lagrange} implies that
\begin{equation*}
\lag A_0 f ,g \rag = \lag f , A_0 g \rag + [f,g](1) = \lag f, A_0 g \rag + f'(1) g(1).
\end{equation*}
Thus, $g(1) = 0$. Symmetric arguments imply that $g(-1)=0$.

\noindent
We now turn to the transmission conditions. Let $f \in D(A)$ be such that its singular part is given by
\begin{equation*}
c_1^+(f) := \frac{1}{2 \nu}, \quad c_2^+(f) := -\frac{1}{2 \nu}.
\end{equation*}
Then, the transmission conditions imply
\begin{equation*}
c_1^-(f) = \frac{1}{2 \nu}, \quad c_2^-(f) = -\frac{1}{2 \nu}.
\end{equation*}
By Lemma \ref{lemme_forme_Lagrange}
\begin{equation*}
\lag A_0 f ,g \rag = \lag f, A_0 g \rag + [f,g](0^-) - [f,g](0^+).
\end{equation*}
Using Lemma \ref{lemme_fr_en_0} it comes that the regular parts have no contribution at $0$ i.e. $[f,g](0^-) =[f_s,g_s](0^-)$ and $[f,g](0^+) = [f_s,g_s](0^+)$. Straightforward computations lead to
\begin{equation*}
[f,g](0^+) = - c_1^+(g) - c_2^+(g), \quad [f,g](0^-) = c_1^-(g) + c_2^-(g).
\end{equation*}
We thus recover the first transmission condition. The second transmission condition follows from similar computations with the choice of a particular $f \in D(A)$ satisfying
\begin{equation*}
c_1^+(f) := - \frac{\nu-\frac{1}{2}}{2 \nu}, \quad c_2^+(f) := - \frac{\nu + \frac{1}{2}}{2 \nu}.
\end{equation*}
Thus, $D(A_0^*) \subset D(A)$. This proves that $(A_n,D(A))$ is a self-adjoint operator.

\medskip
\textit{Fourth step : positivity.} We end the proof of Proposition \ref{prop_auto_adjoint} by proving (\ref{coercivite_1d}). Let $f \in D(A)$.

\noindent
Using Lemma \ref{lemme_forme_Lagrange} and integration by parts it comes that
\begin{align*}
&\lag A_n f, f \rag = \int_{-1}^1 \Big(-\partial^2_{xx} f_r + \frac{c_\nu}{x^2} f_r \Big)(x) f(x) \md x
+ \int_{-1}^1 (n \pi)^2 |x|^{2 \gamma} f^2(x) \md x,
\\
&= \int_{-1}^1 (\partial_x f_r)^2(x) + \frac{c_\nu}{x^2} f_r^2(x) \md x + \int_{-1}^1 (n \pi)^2 |x|^{2 \gamma} f^2(x) \md x + (-\partial_x f_r)(1) f_r(1)
\\
& + \partial_x f_r(-1) f_r(-1) + [f_r, f_s](1) - [f_r, f_s](0^+) + [f_r,f_s](0^-) - [f_r,f_s](-1).
\end{align*}
Using Lemma \ref{lemme_fr_en_0}, it comes that $[f_r,f_s](0^+)=[f_r,f_s](0^-)=0$. Gathering the boundary terms and using $f(1)=f(-1)=0$ it comes that
\begin{align*}
\lag A_n f , f\rag &= \int_{-1}^1 (\partial_x f_r)^2(x) + \frac{c_\nu}{x^2} f_r^2(x) \md x + \int_{-1}^1 (n \pi)^2 |x|^{2 \gamma} f^2(x) \md x
\\
& + f_r(1) \partial_x f_s (1) - f_r(-1) \partial_x f_s(-1).
\label{positivite_inter}
\tag{\theequation} \addtocounter{equation}{1}
\end{align*}
As $f(1)=f(-1)=0$, it comes that
\begin{gather*}
f_r(1) \partial_x f_s (1) = - \big( c_1^+(f) + c_2^+(f) \big) \left( \left(\nu+ \frac{1}{2}\right) c_1^+(f) + \left(-\nu + \frac{1}{2} \right) c_2^+(f) \right),
\\
f_r(-1) \partial_x f_s (-1) =  \big( c_1^-(f) + c_2^-(f) \big) \left( \left(\nu+ \frac{1}{2}\right) c_1^-(f) + \left(-\nu + \frac{1}{2} \right) c_2^-(f) \right).
\end{gather*}
Thus, a sufficient condition to ensure that $A_n$ is non-negative is
\begin{align*} 
&\big( c_1^-(f) + c_2^-(f) \big) \left( \left(\nu+ \frac{1}{2}\right) c_1^-(f) + \left(-\nu + \frac{1}{2} \right) c_2^-(f) \right)  
\\
&=  - \big( c_1^+(f) + c_2^+(f) \big) \left( \left(\nu+ \frac{1}{2}\right) c_1^+(f) + \left(-\nu + \frac{1}{2} \right) c_2^+(f) \right).
\label{CS_positivite}
\tag{\theequation} \addtocounter{equation}{1}
\end{align*}
This follows directly from the transmission conditions.
Thus, (\ref{positivite_inter}) implies
\begin{equation} \label{positivite_A}
\lag Af , f\rag \geq  \int_{-1}^1 (\partial_x f_r)^2(x) + \frac{c_\nu}{x^2} f_r^2(x) \md x + \int_{-1}^1 (n \pi)^2 |x|^{2 \gamma} f^2(x) \md x.
\end{equation}
If $c_\nu \geq 0$, we get (\ref{coercivite_1d}) with $m_\nu =1$. If $c_\nu < 0$, using Hardy's inequality (\ref{Hardy_1d_(-1,1)}), it comes that
\begin{align*}
&\int_{-1}^1 (\partial_x f_r)^2(x) + \frac{c_\nu}{x^2} f_r^2(x) \md x 
\\
&= \left( 1 + 4 c_\nu \right) \int_{-1}^1 (\partial_x f_r)^2(x) \md x - 4 c_\nu \int_{-1}^1 \left( (\partial_x f_r)^2(x) - \frac{1}{4} \frac{f_r^2(x)}{x^2} \right) \md x 
\\
&\geq \left( 1 + 4 c_\nu \right) \int_{-1}^1 (\partial_x f_r)^2(x) \md x.
\end{align*}
This gives (\ref{coercivite_1d}) with $m_\nu = 4 \nu^2$. 
This ends the proof of Proposition~\ref{prop_auto_adjoint}.

\end{proof}

\subsection{Semigroup associated to the \textsc{2d} problem}
\label{subsect_bien_pose_2d}

Let $f^0 \in L^2(\Omega)$. For almost every $x \in (-1,1)$, $f^0(x,\cdot) \in L^2(0,1)$ and thus can be expanded in Fourier series as follows
\begin{equation} \label{dec_fourier_f0}
f^0(x,y) = \sum_{n \in \N^*} f_n^0(x) \varphi_n(y),
\end{equation}
where $(\varphi_n)_{n \in \N^*}$ is the Hilbert basis of $L^2(0,1)$ of eigenvectors of the Laplace operator on $H^2(0,1)$ with homogeneous boundary conditions i.e.
\begin{equation*}
\varphi_n (y) := \sqrt{2} \sin(n \pi y), \quad \forall n \in \N^*, 
\end{equation*}
and
\begin{equation*}
f_n^0(x) := \int_{-1}^1 f^0(x,y) \varphi_n(y) \md y.
\end{equation*}

\noindent
For any $t \in (0,T)$, we define the following operator
\begin{equation} \label{def_semigroupe_2d}
(S(t) f^0)(x,y) := \sum_{n \in \N^*} f_n(t,x) \varphi_n(y),
\end{equation}
where for any $n \in \N^*$, $f_n(t) := e^{-A_n t} f_n^0$. Then, the following proposition holds.
\begin{prop} \label{prop_semigroupe_2d}
$S(t)$ defined by (\ref{def_semigroupe_2d}) is a continuous semigroup of contraction in $L^2(\Omega)$.
\end{prop}

\begin{proof}[Proof of Proposition \ref{prop_semigroupe_2d}] By Proposition \ref{prop_bien_pose_1d}, $S(t)$ is well defined, with value in $L^2(\Omega)$, it is a semigroup and satisfies the contraction property.
For any $f^0 \in L^2(\Omega)$, we have
\begin{equation*}
|| S(t) f^0 - f^0 ||_{L^2(\Omega)}^2 = \sum_{n \in \N^*} || f_n(t,\cdot) - f_n^0 ||_{L^2(-1,1)}^2.
\end{equation*}
By Proposition \ref{prop_bien_pose_1d} it comes that 
\begin{gather*}
|| f_n(t,\cdot) - f_n^0 ||_{L^2(-1,1)} \underset{t \to 0}{\longrightarrow} 0,
\\
|| f_n(t,\cdot) - f_n^0 ||_{L^2(-1,1)} \leq 2 || f_n^0 ||_{L^2(-1,1)}.
\end{gather*} 
Thus, by the dominated convergence theorem, $S(t) f^0 \underset{t \to 0}{\longrightarrow} f^0$ in $L^2(\Omega)$.

\end{proof}

\noindent
Recall that the infinitesimal generator $\AA$ of $S(t)$ is defined on
\begin{equation*}
D(\AA) := \left\{ f \in L^2(\Omega) \, ; \, \lim\limits_{t \to 0} \frac{S(t)f - f}{t} \text{ exists } \right\},
\end{equation*}
by
\begin{equation*}
\AA f := \lim\limits_{t \to 0} \frac{S(t)f - f}{t}.
\end{equation*}
The previous limits are related to the $L^2$ norm.
Then, from \cite[Theorems 1.3.1 and 1.4.3]{PazyBook} it comes that $(\AA, D(\AA))$ is a closed dissipative densely defined operator and satisfies for any $\lambda >0$, $R(\lambda I - \AA) = L^2(\Omega)$. The following proposition links the system (\ref{grushin_2d}) and the semigroup $S(t)$.

\begin{prop} \label{prop_lien_2d_1d}
The infinitesimal generator $\AA$ of $S(t)$ is characterized by
\begin{align*} 
D(\AA) = \bigg\{ f \in L^2(\Omega) \, ; \, &f = \sum_{n \in \N^*} f_n(x) \varphi_n(y) \text{ with } f_n \in D(A)
\text{ and }
\\
& \sum_{n \in \N^*} ||A_n f_n||_{L^2(-1,1)}^2 < + \infty \bigg\},
\tag{\theequation} \addtocounter{equation}{1}
\label{domaine_operateur_2d}
\end{align*}
and
\begin{equation} \label{operateur_2d}
\AA f = - \sum_{n \in \N^*} (A_n f_n)(x) \varphi_n(y).
\end{equation}
This operator extends the Grushin differential operator in the sense that
\begin{equation} \label{operateur_2d_extension}
\AA f = \partial^2_{xx} f + |x|^{2 \gamma} \partial^2_{yy} f - \frac{c_\nu}{x^2} f, \quad \forall f \in C_0^\infty(\Omega \backslash \{x=0\}).
\end{equation}
\end{prop}

\begin{proof}[Proof of Proposition \ref{prop_lien_2d_1d}]
Let $f^0  \in D(\AA)$. Then, $\AA f^0 \in L^2(\Omega)$ and
\begin{equation*}
\frac{S(t) f^0 - f^0}{t} \underset{t \to 0}{\longrightarrow} \AA f^0, \quad \text{in } L^2(\Omega).
\end{equation*}
As $\AA f^0 \in L^2(\Omega)$, it can be decomposed in Fourier series in the $y$ variable i.e.
\begin{equation*}
\AA f^0 (x,y) = \sum_{n \in \N^*} (\AA f^0)_n (x) \varphi_n(y).
\end{equation*}
Thus,
\begin{equation*}
\left| \left| \frac{S(t) f^0 - f^0}{t} - \AA f^0 \right| \right|_{L^2(\Omega)}^2 
= \sum_{n \in \N^*} \left| \left| \frac{f_n(t) - f_n^0}{t} - (\AA f^0)_n \right| \right|_{L^2(-1,1)}^2
\underset{t \to 0}{\longrightarrow} 0.
\end{equation*}
This implies that for any $n \in \N^*$, $f_n^0 \in D(A)$ and
\begin{equation*}
(\AA f^0)_n = - A_n f_n^0.
\end{equation*}
We thus get
\begin{equation*}
-\AA f^0 = \sum_{n \in \N^*} (A_n f_n^0)(x) \varphi_n(y).
\end{equation*}
Conversely, let $g \in L^2(\Omega)$ be such that for any $n \in \N^*$, $g_n \in D(A)$ and 
$\sum\limits_{n \in \N^*} ||A_n g_n ||^2_{L^2(-1,1)} < + \infty$. Let $f \in D(\AA)$.
Then,
\begin{equation*}
| \lag \AA f , g \rag | \leq \sum_{n \in \N^*} |\lag A_n f_n ,g_n \rag| 
\leq \left( \sum_{n \in \N^*} ||f_n||_{L^2}^2 \right)^{\frac{1}{2}}
\left( \sum_{n \in \N^*} ||A_n g_n||_{L^2}^2 \right)^{\frac{1}{2}}.
\end{equation*}
This implies that $g \in D(\AA^*)$. Finally, self-adjointness of $S(t)$ and thus of $\AA$ ends the proof of (\ref{domaine_operateur_2d}).
Straightforward computations lead to (\ref{operateur_2d_extension}) and thus ends the proof of Proposition \ref{prop_lien_2d_1d}.

\end{proof}


\noindent
Using Proposition \ref{prop_lien_2d_1d}, we rewrite (\ref{grushin_2d})-(\ref{grushin_2d_ci}) in the form
\begin{equation} \label{syst_2d_abstrait}
\left\{
\begin{aligned}
&f'(t) = \AA f(t) + v(t),  \quad &t& \in [0,T],
\\
&f(0) = f^0,
\end{aligned}
\right.
\end{equation}
where $v(t) : (x,y) \in \Omega \mapsto u(t,x,y) \chi_{\omega}(x,y)$. The following proposition is classical (see e.g. \cite{PazyBook}) and ends this well posedness section
\begin{prop} \label{prop_bien_pose_2d}
For any $f^0 \in L^2(\Omega)$, $T>0$ and $v \in L^1((0,T);L^2(\Omega))$, system (\ref{syst_2d_abstrait}) has a unique mild solution given by
\begin{equation*}
f(t) = S(t) f^0  +  \int_0^t S(t-\tau) v(\tau) \md \tau, \quad t \in [0,T].
\end{equation*}
\end{prop}
In the following a solution of (\ref{grushin_2d}) will mean a solution of (\ref{syst_2d_abstrait}).

\subsection{General theory of self-adjoint extensions}
\label{subsect_extensions}

This subsection is dedicated to enlighten the choices made in the construction of the functional setting leading to the definition (\ref{def_D(A)}) of $D(A)$.

The question of finding the self-adjoint extensions of a given closed symmetric operator is classical. In \cite[Theorem X.2]{ReedSimon2} such extensions are characterized by means of isometries between the deficiency subspaces. 
The particular case of Sturm-Liouville operators has been widely studied : most of these result are contained in \cite{ZettlBook}. The self-adjoint extensions are characterized by means of generalized boundary conditions.
In our case, we are concerned with the Sturm-Liouville operator $-\frac{\md^2}{\md x^2} + \frac{c_\nu}{x^2}$ on the interval $(-1,1)$. This fits in the setting of \cite[Chapter 13]{ZettlBook}. The number of boundary conditions to impose is given by the deficiency index. Following \cite[Proposition 3.1]{AlekseevaAnanieva}, it comes that our operator on the interval $(0,1)$ has deficiency index $2$. This is closely related to the fact that $\nu \in (0,1)$. Then, \cite[Lemma 13.3.1]{ZettlBook} implies that the deficiency index for the interval $(-1,1)$ is $4$. We thus get the following proposition which is simply a rewriting of \cite[Theorem 13.3.1 Case 5]{ZettlBook}.

\begin{prop} \label{prop_extensions_Zettl}
Let $u$ and $v$ in $D_{max}$ be such that their restriction on $(0,1)$ (resp. $(-1,0)$) are linearly independent modulo $H^2_0(0,1)$ (resp. $H^2_0(-1,0)$) and
\begin{equation*}
[u,v](-1) = [u,v] (0^-) = [u,v](0^+) = [u,v](1) = 1.
\end{equation*}
Let $M_1, \dots, M_4$ be $4 \times 2$ complex matrices. Then every self-adjoint extension of the minimal operator is given by the restriction of $D_{max}$ to the functions $f$ satisfying the boundary conditions
\begin{equation*}
M_1 \begin{pmatrix} [f,u](-1) \\ [f,v](-1) \end{pmatrix}  + 
M_2 \begin{pmatrix} [f,u](0^-) \\ [f,v](0^-) \end{pmatrix}  + 
M_3 \begin{pmatrix} [f,u](0^+) \\ [f,v](0^+) \end{pmatrix}  + 
M_4 \begin{pmatrix} [f,u](1) \\ [f,v](1) \end{pmatrix}  = 0,
\end{equation*}
where the matrices satisfy $(M_1 \, M_2 \, M_3 \, M_4)$ has full rank and 
\begin{equation*}
M_1 E M_1^* - M_2 E M_2^* + M_3 E M_3^* - M_4 E M_4^* = 0, 
\text{ with } E := \begin{pmatrix} 0 & -1 \\ 1 & 0 \end{pmatrix}.
\end{equation*}
Conversely, every choice of such matrices defines a self-adjoint extension.
\end{prop}

\noindent
We end this section by giving the choice of such matrices that we made and give another functional setting that would lead to well posedness but that is not adapted to controllability issues. We define on $(0,1)$ $u$ and $v$ to be solutions of 
\begin{equation*}
-f''(x) + \frac{c_\nu}{x^2} f(x) = 0
\end{equation*}
with $(u(1) = 0, u'(1)=1)$ and $(v(1)=-1, v'(1)=0)$ i.e.
\begin{align*}
u(x) &= \frac{1}{2\nu} x^{\nu+ \frac{1}{2}} - \frac{1}{2\nu} x^{-\nu+\frac{1}{2}}, 
\\
v(x) &= - \frac{\nu-\frac{1}{2}}{2\nu} x^{\nu+\frac{1}{2}} - \frac{\nu+\frac{1}{2}}{2\nu} x^{-\nu+\frac{1}{2}}.
\end{align*}
Thus for any $f \in D_{max}$, $[f,u](1)=f(1)$ and $[f,v](1)=f'(1)$, and for any $x \in [0,1]$, $[u,v](x) \equiv 1$.
We design $u$ and $v$ similarly on $(-1,0)$ i.e.
\begin{align*}
u(x) &= -\frac{1}{2\nu} |x|^{\nu+\frac{1}{2}} + \frac{1}{2\nu} |x|^{-\nu+\frac{1}{2}}, 
\\
v(x) &= - \frac{\nu-\frac{1}{2}}{2\nu} |x|^{\nu+\frac{1}{2}} - \frac{\nu+\frac{1}{2}}{2\nu} |x|^{-\nu+\frac{1}{2}}.
\end{align*}

\noindent
Due to the choice of functions $u$ and $v$, the homogeneous Dirichlet conditions at $\pm 1$ are implied by the choice
\begin{equation*}
M_1 = \begin{pmatrix} 1 & 0 \\ 0 & 0 \\ 0 & 0 \\ 0 & 0 \end{pmatrix}, \:
M_2 = \begin{pmatrix} 0 & 0 \\ \multicolumn{2}{c}{\multirow{2}{*}{$\tilde{M}_2$}} \\ & \\ 0 & 0 \end{pmatrix}, \:
M_3 = \begin{pmatrix} 0 & 0 \\ \multicolumn{2}{c}{\multirow{2}{*}{$\tilde{M}_3$}} \\ & \\ 0 & 0 \end{pmatrix}, \:
M_4 = \begin{pmatrix} 0 & 0 \\ 0 & 0 \\ 0 & 0 \\ 1 & 0 \end{pmatrix}.
\end{equation*}

\noindent
Then, the conditions of Proposition \ref{prop_extensions_Zettl} are satisfied if and only if the matrix $(\tilde{M}_2 \, \tilde{M}_3)$ has rank $2$ and $\det(\tilde{M}_2) = \det(\tilde{M}_3)$. Straightforward computations lead to, for any $f \in D_{max}$
\begin{align*}
[f,u](0^+) &= c_1^+ + c_2^+, \qquad 
[f,v](0^+) = \left(\nu+\frac{1}{2} \right)  c_1^+ + \left(-\nu + \frac{1}{2} \right) c_2^+,
\\
[f,u](0^-) &= c_1^- + c_2^-, \qquad 
[f,v](0^-) = -\left(\nu+\frac{1}{2} \right)  c_1^- - \left(-\nu + \frac{1}{2} \right) c_2^-.
\end{align*}

\medskip
\textbf{Construction of $D(A)$.}
The choice 
\begin{equation*}
\tilde{M}_2 = \tilde{M}_3 = \begin{pmatrix} 1 & 0 \\ 0 & 1 \end{pmatrix}
\end{equation*} 
lead to the definition of $D(A)$ in (\ref{def_D(A)}). The computations done in the fourth step of the proof of Proposition~\ref{prop_auto_adjoint} (see (\ref{positivite_A})) prove the positivity and thus, Proposition~\ref{prop_auto_adjoint} could also be seen as an application of Proposition~\ref{prop_extensions_Zettl}.

\medskip
\textbf{Other construction.}
At this stage, there is another choice that would lead to a self-adjoint positive extension. If, we set 
\begin{equation*}
\tilde{M}_2 =  \begin{pmatrix} 0 & 0 \\ 0 & 1 \end{pmatrix} \text{ and } 
\tilde{M}_3 = \begin{pmatrix} 1 & 0 \\ 0 & 0 \end{pmatrix},
\end{equation*}
then the domain with conditions
\begin{equation} \label{mauvaises_cond_transmission}
c_1^+ = - c_2^+,\quad 
c_1^- = - \frac{-\nu+\frac{1}{2}}{\nu+\frac{1}{2}} c_2^-,
\end{equation}
give rise to a self-adjoint positive operator. However, from a point of view of controllability, this domain does not seem interesting as this conditions couple the coefficients on each side on the singularity. As it can be noticed from the proof of Proposition~\ref{prop_uc_1er_cote}, once the domain of $A_n$ is defined to ensure well-posedness, the only requirement on the transmission condition to obtain the approximate controllability result of Theorem~\ref{th_controle_approche_2d} is
\begin{equation*}
c_1^- = c_2^- = 0 \quad \Longrightarrow \quad c_1^+ = c_2^+ = 0.
\end{equation*}
This is not satisfied for the transmission conditions (\ref{mauvaises_cond_transmission}) and we cannot apply the results developed in this article to this functional setting.

As a matter of fact, one gets from \cite[Proposition 10.4.2]{ZettlBook} (characterizing self-adjoint extensions in the case of a boundary singularity, similarly to Proposition~\ref{prop_extensions_Zettl} for internal singularity) that the two problems on $(-1,0)$ and $(0,1)$ with conditions (\ref{mauvaises_cond_transmission}) are well-posed. Thus the dynamics really are decoupled and approximate controllability from one side of the singularity does not hold for the transmission conditions (\ref{mauvaises_cond_transmission}).

\begin{rmq}
Notice that the goal here is not to give an exhaustive characterization. This will be pointless with regards to the main goal of giving a meaning to (\ref{grushin_2d}). Indeed, in the construction of the semigroup $S$ we here imposed the same transmission conditions for each Fourier component. As soon as we have different transmission conditions ensuring to have a self-adjoint extension $A_n$, there is infinitely many extensions $\AA$ generating a semigroup. 
\end{rmq}

\medskip
\textbf{Symmetry of transmission.}
At this stage, one can wonder if there are non-symmetric transmission conditions i.e. a choice of matrices $\tilde{M}_2$ and $\tilde{M}_3$ such that
\begin{gather*}
c_1^- = c_2^- = 0 \quad \Longrightarrow \quad c_1^+ = c_2^+ = 0,
\\
c_1^+ = c_2^+ = 0 \quad \hspace{4pt}\not\hspace{-4pt}\Longrightarrow \quad c_1^- = c_2^- = 0.
\end{gather*}
Rewriting the condition
\begin{equation*}
\tilde{M}_2 \begin{pmatrix} [f,u](0^-) \\ [f,v](0^-) \end{pmatrix}  + 
\tilde{M}_3 \begin{pmatrix} [f,u](0^+) \\ [f,v](0^+) \end{pmatrix}  
= \begin{pmatrix} 0 \\ 0 \end{pmatrix}
\end{equation*}
as
\begin{equation*}
\hat{M}_2 \begin{pmatrix} c_1^- \\ c_2^- \end{pmatrix}  + 
\hat{M}_3 \begin{pmatrix} c_1^+ \\ c_2^+ \end{pmatrix}  
= \begin{pmatrix} 0 \\ 0 \end{pmatrix}
\end{equation*}
we get that the condition $\det(\tilde{M}_2) = \det(\tilde{M}_3)$ implies $\det(\hat{M}_2) = \det(\hat{M}_3)$. It is then not possible to have only one of the matrices $\hat{M}_2$ and $\hat{M}_3$ invertible. Thus, in this setting, there is no choice of non-symmetric transmission conditions.

\section{Unique continuation}
\label{sect_continuation_unique}

We consider the adjoint system of (\ref{grushin_2d})
\begin{equation} \label{grushin_2d_adjoint}
\left\{
\begin{aligned}
& \partial_t g - \partial^2_{xx} g - |x|^{2 \gamma} \partial^2_{yy} g + \frac{c_\nu}{x^2} g = 0, \; &(t,x,y)& \in (0,T) \times \Omega,
\\
&g(t,x,y) = 0, \; &(t,x,y)& \in (0,T) \times \partial \Omega,
\\
&g(0,x,y) = g^0(x,y), \; &(x,y)& \in \Omega.
\end{aligned}
\right.
\end{equation}
Here again this system is understood in the sense of the self-adjoint extension $\AA$ designed i.e. for any $g^0 \in L^2 (\Omega)$, the solution of (\ref{grushin_2d_adjoint}) is given by $S(t) g^0$.

This section is dedicated to the study of the following unique continuation property
\begin{defi} \label{Def:UC}
Let $T>0$ and $\omega \subset \Omega$. We say that the unique continuation property from $\omega$ holds for system (\ref{grushin_2d_adjoint}) if the only solution of (\ref{grushin_2d_adjoint}) vanishing on $(0,T) \times \omega$ is identically zero on $(0,T) \times \Omega$ i.e.
\begin{equation*}
\Big( g^0 \in L^2(\Omega), \,  \chi_{\omega} S(t) g^0 = 0 \text{ for a.e. } t \in (0,T)  \Big) 
\Longrightarrow  g^0 =0.
\end{equation*}
\end{defi}

By a classical duality argument, Theorem~\ref{th_controle_approche_2d} is equivalent to the following unique continuation theorem.
\begin{theo} \label{th_continuation_unique}
Let $T>0$, $\gamma > 0$ and $\nu \in (0,1)$. Let $\omega$ be an open subset of $\Omega$. The unique continuation property from $\omega$ holds for the adjoint system (\ref{grushin_2d_adjoint}) in the sense of Definition~\ref{Def:UC}.
\end{theo}
Without loss of generality, we may assume that $\omega$ is an open subset of one connected component of $ \Omega \backslash \{ x= 0\}$. As it will be noticed in  Remark~\ref{rk_controle_deux_cotes}, if $\omega$ intersects both connected components of $\Omega \backslash \{ x = 0 \}$ then the proof is simpler.
In the following we assume that $\omega \subset (-1,0) \times (0,1)$.

The rest of this section is dedicated to the proof of Theorem~\ref{th_continuation_unique}. In Sect.~\ref{subsect_uc_1}, we prove that if $g(t):= S(t) g^0$ is vanishing on $(0,T) \times \omega$ then it is vanishing on $(0,T) \times (-1,0) \times (0,1)$. This will imply that any Fourier component $g_n$ has no singular part and is identically zero on $[-1,0]$. 

Then, we are left to study a one dimensional equation on the regular part with a boundary inverse square singularity. Dealing with the regular part, we know furthermore that the function under study has the $H^2$ regularity and satisfies $\partial_x g_n(t,0) =0$. This will be used in Sect.~\ref{subsect_carleman} to prove a suitable Carleman estimate, relying on an adapted Hardy inequality, to end the proof of Theorem~\ref{th_continuation_unique}.

\subsection{Reduction to the case of a boundary singularity}
\label{subsect_uc_1}

The goal of this section is the proof of the following proposition
\begin{prop} \label{prop_uc_1er_cote}
Let $T>0$, $\gamma > 0$, $\nu \in (0,1)$ and $\omega$ be an open subset of $(-1,0) \times (0,1)$. Assume that $g^0 \in L^2(\Omega)$ is such that $ g(t) := S(t) g^0$ is vanishing on $(0,T) \times \omega$. Then $g$ is vanishing on $(0,T) \times (-1,0) \times (0,1)$. Moreover, for any $n \in \N^*$, the $n^{th}$ Fourier component satisfies
\begin{gather*}
c_1^-(g_n) = c_2^-(g_n) = c_1^+(g_n) = c_2^+(g_n) =0,
\\
g_n(t,x) = \chi_{(0,1)}(x) g_{n,r}(t,x) \text{ for every } (t,x) \in (0,T) \times (-1,1),
\end{gather*}
where $g_{n,r}$ is the regular part of $g_n$.
\end{prop}

\smallskip
\noindent
\begin{proof}[Proof of Proposition \ref{prop_uc_1er_cote}] 
Let $\varepsilon > 0$ be such that
\begin{equation*}
\omega \subset \Omega^-_{\varepsilon} := (-1,-\varepsilon) \times (0,1).
\end{equation*}
For every $t \in [0,T]$,
\begin{equation*}
(S(t) g^0)(x,y) = \sum_{n \in \N^*} g_n(t,x) \varphi_n(y),
\end{equation*}
where $g_n$ is the solution of (\ref{grushin_1d}) with initial condition $g_n^0$.

\noindent
We check that on $\Omega^-_{\varepsilon}$, the operator $\AA$ is uniformly elliptic. Let $h \in D(\AA)$ and $\phi \in C^\infty_0(\Omega_\varepsilon^-)$. Then,
\begin{align*}
\lag \AA h , \phi \rag_{L^2(\Omega_\varepsilon^-)} &=
\int_{-1}^{-\varepsilon} \int_0^1 \AA h (x,y) \phi(x,y) \md y \md x
\\
&= - \sum_{n \in \N^*} \lag A_n h_n , \phi_n \rag_{L^2(-1,-\varepsilon)}
\\
&= - \sum_{n \in \N^*} \lag h_n , A_n \phi_n \rag_{L^2(-1,-\varepsilon)}
\\
&= \lag h , \left(\partial^2_{xx} + |x|^{2 \gamma} \partial^2_{yy} - \frac{c_\nu}{x^2} \right) \phi \rag_{L^2(\Omega_\varepsilon^-)}.
\end{align*}
Thus, $h \in D(\AA)$ implies that
\begin{equation*}
\AA h \overset{\mathcal{D}'(\Omega_\varepsilon^-)}{=} \left(\partial^2_{xx} + |x|^{2 \gamma} \partial^2_{yy} - \frac{c_\nu}{x^2} \right) h.
\end{equation*}

As $h \in D(\AA)$, this equality also holds in $L^2(\Omega_\varepsilon^-)$. In particular, this implies that
\begin{equation*}
\partial^2_{xx} h + |x|^{2 \gamma} \partial^2_{yy} h \in L^2(\Omega_\varepsilon^-),
\end{equation*}
and also that $\AA$ is uniformly elliptic on $\Omega_\varepsilon^-$. 
Thus, using classical unique continuation  results for uniformly parabolic operators with variable coefficients (see e.g. \cite[Theorem 1.1]{SautScheurer87}), it comes that $S(t) g^0 =0$ for every $t \in (0,T]$ in $L^2(\Omega_\varepsilon^-)$. Then, it comes that $S(t) g^0 =0$ for every $t \in (0,T]$ in $L^2(\Omega_0^-)$. If, for any $n \in \N^*$, we decompose $g_n$ in regular and singular part (as defined in (\ref{def_D(A)})) we get
\begin{equation} \label{nulle_1er_cote_singulier}
c_1^-(g_n(t)) = c_2^-(g_n(t)) = 0, \quad \forall t \in (0,T),
\end{equation}
\begin{equation} \label{nulle_1er_cote_regulier}
g_{n,r}(t,x) = 0, \quad \forall (t,x) \in (0,T) \times (-1,0).
\end{equation}
Using the transmission conditions in (\ref{def_D(A)}), it also comes that $c_1^+(g_n(t)) = c_2^+(g_n(t)) =0$ and thus the singular part is identically zero on $(0,T) \times (-1,1)$. This ends the proof of Proposition~\ref{prop_uc_1er_cote}.

\end{proof}

\begin{rmq} \label{rk_controle_deux_cotes}
Notice that Proposition \ref{prop_uc_1er_cote} proves that if $\omega$ intersects both connected components of $\Omega \backslash \{ x=0\}$, then unique continuation from $\omega$ hold for any $\nu \in (0,1)$.
\end{rmq}
\medskip

Proposition~\ref{prop_uc_1er_cote} implies that if $\chi_\omega S(t) g^0$ is identically zero then for any $n \in \N^*$, $g_n \in C^0((0,T],H^2 \cap H^1_0 (0,1)) \cap C^1((0,T],L^2(0,1))$ is solution of
\begin{equation} \label{grushin_1d_1cote}
\left\{
\begin{aligned}
&\partial_t g_n - \partial^2_{xx} g_n + \left( \frac{c_\nu}{x^2} + (n \pi)^2 x^{2 \gamma} \right) g_n = 0, \, &(t,x)& \in (0,T) \times (0,1),
\\
& g_n(t,0) = g_n(t,1) = 0, \, &t& \in (0,T),
\\
& \partial_x g_n(t,0) = 0, \, &t& \in (0,T).
\end{aligned}
\right.
\end{equation}
We prove in Sect.~\ref{subsect_carleman} that this leads to $g_n \equiv 0$ using a suitable Carleman estimate.

\subsection{Carleman estimate}
\label{subsect_carleman}

This subsection is dedicated to the proof of the following Carleman type inequality. 

\begin{prop} \label{prop_Carleman}
Let $T>0$ and $Q_T := (0,T) \times (0,1)$. 
There exist $R_0, C_0 > 0$ such that for any $R \geq R_0$, for every $\gamma >0$, $\nu \in (0,1)$ and $n \in \N^*$, any $g \in C^1((0,T], L^2(0,1)) \cap C^0((0,T], H^2 \cap H^1_0(0,1))$ with $\partial_x g(t,0) \equiv 0$ on $(0,T)$ satisfies
\begin{align*} 
&C_0 R^3 \iint_{Q_T} \frac{1}{(t(T-t))^3} \exp \left( \frac{-2 R x^b}{t(T-t)} \right) g^2(t,x) \md x \md t 
\\
&\leq \iint_{Q_T} | \PP_{n,\nu,\gamma} g(t,x)|^2 \exp \left( \frac{-2 R x^b}{t(T-t)} \right) \md x \md t,
\tag{\theequation} \addtocounter{equation}{1}
\label{Carleman}
\end{align*}
where
\begin{equation*}
\PP_{n,\nu,\gamma} := \partial_t  - \partial^2_{xx}  + \left( \frac{c_\nu}{x^2} + (n \pi)^2 x^{2 \gamma} \right),
\end{equation*}
and $b$ satisfies 
\begin{equation} \label{condition_b}
\left\{ 
\begin{aligned}
&\text{ if } \nu \in \left( 0 ,\tfrac{1}{2} \right], &b& \in (0,1),
\\
&\text{ if } \nu \in \left( \tfrac{1}{2}, 1 \right),  &b& := 2 - 2 \nu \in (0,1).
\end{aligned}
\right.
\end{equation}
\end{prop}

\noindent
Before proving Proposition \ref{prop_Carleman} we show that it ends the proof of Theorem~\ref{th_continuation_unique}. Let $g^0 \in L^2(\Omega)$ be such that $\chi_\omega S(t) g^0 \equiv 0$. Using Proposition \ref{prop_uc_1er_cote} and the final comment of Sect.~\ref{subsect_uc_1}, it comes that for any $n \in \N^*$, $g_n \in C^1((0,T], L^2(0,1)) \cap C^0((0,T], H^2 \cap H^1_0(0,1))$ with $\partial_x g_n(t,0) \equiv 0$ on $(0,T)$. As, $g_n$ is solution of (\ref{grushin_1d_1cote}), it comes that $\PP_{n,\nu,\gamma} g_n \equiv 0$ on $(0,T) \times (0,1)$. Then, Proposition \ref{prop_Carleman} implies that $g_n \equiv 0$ and thus, as $g_n \in C^0([0,T], L^2(0,1))$, we recover $g^0 = 0$.

\begin{rmq} 
Contrarily to Carleman estimates proved by Vancostenoble \cite{Vancostenoble11}, there are no boundary terms in the right-hand side of the inequality. Actually, the homogeneous Neumann boundary condition at $x=0$  is crucial for inequality (\ref{Carleman}) to hold.
\end{rmq}
\smallskip

The proof will rely on the following Hardy type inequality.
\begin{prop} \label{Prop:Hardy}
For any $z \in H^2 \cap H^1_0 (0,1)$ with $z'(0) = 0$,
\begin{equation*}
\frac{(1-\alpha)^2}{4} \int_0^1 x^{\alpha-2} z(x)^2 \md x \leq \int_0^1 x^\alpha z'(x)^2 \md x  < \infty, 
\quad \forall \alpha \in [-2, 2).
\end{equation*}
\end{prop}

The statement and the proof are classical (see for example \cite[Theorem 2.1]{Vancostenoble11}). The main novelty here is that due to the extra information $z'(0)=0$, we can prove the Hardy inequality with singular potential up to $\alpha \geq -2$.
\begin{proof}[Proof of Proposition~\ref{Prop:Hardy}]
Applying the generalized Hardy inequality \cite{CannarsaMartinezVancostenoble08} we get Proposition~\ref{Prop:Hardy} for $\alpha \in [0,2)$. Let $\alpha \in [-2,0)$. Applying this generalized Hardy inequality to $z'$, we have
\begin{equation*}
\frac{(\alpha+1)^2}{4} \int_0^1 x^\alpha z'(x)^2 \md x \leq  \int_0^1 x^{\alpha+2} z''(x)^2 \md x < \infty.
\end{equation*}
For any $c \in \R$,
\begin{align*}
0 &\leq \int_0^1 \left( x^{\frac{\alpha}{2}} z'(x)  + c x^{\frac{\alpha-2}{2}} z(x) \right)^2 \md x
\\
&= \int_0^1 x^\alpha z'(x)^2  + c^2 x^{\alpha-2} z(x)^2 + c x^{\alpha-1} (z^2)'(x) \md x.
\end{align*}
From Lemma~\ref{lemme_fr_en_0}, integration by parts lead to 
\begin{equation*}
\int_0^1  x^{\alpha -1} (z^2)'(x) \md x = - \int_0^1 (\alpha - 1) x^{\alpha-2} z(x)^2 \md x.
\end{equation*}
Thus, for any $c \in \R$,
\begin{equation*} 
\int_0^1 x^\alpha z'(x)^2 \md x \geq \left( c (\alpha-1) - c^2 \right) \int_0^1 x^{\alpha-2} z^2 \md x.
\end{equation*}
Choosing $c = \tfrac{1}{2} (\alpha-1)$ ends the proof of Proposition~\ref{Prop:Hardy}.

\end{proof}

We set some notations that will used throughout the proof.

\noindent
Let $\theta : t \in (0,T) \mapsto \frac{1}{t(T-t)}$. Let $\sigma (t,x) := \theta(t) x^b$ where $b$ satisfies (\ref{condition_b}). Notice that as $b \in (0,1)$ every space derivative of the weight function is singular at $x=0$. This will be useful to handle the singular potential. The idea of using such a weight is inspired by~\cite{CannarsaTortYamamoto12}. In Remark~\ref{rk_poids_Carleman2}, we give further comments on this choice of weight function. To simplify the notations, we denote the partial derivatives by subscripts: $g_x$ stands for $\partial_x g$.

\begin{proof}[Proof of Proposition \ref{prop_Carleman}]
We set for $R > 0$,
\begin{equation} \label{def_z}
z(t,x) := e^{-R \sigma(t,x)} g(t,x).
\end{equation}
From the definition of $\sigma$ we get that, for any $x \in (0,1)$, $z(0,x) = z(T,x) = z_t(0,x) = z_t(T,x) = 0$. The boundary conditions on $g$ also imply that for any $t \in(0,T)$, $z(t,0) = z(t,1) = z_x(t,0) = 0$.

By Proposition~\ref{Prop:Hardy}, these boundary conditions imply that $x \mapsto\frac{z(t,x)}{x^2} \in L^2(0,1)$, $x \mapsto \frac{z_x(t,x)}{x} \in L^2(0,1)$ and using Lemma~\ref{lemme_fr_en_0} we get
\begin{equation} \label{limites_z}
\frac{z(t,x)}{x^{\frac{3}{2}}} \underset{x \to 0}{\longrightarrow} 0, \qquad
\frac{z_x(t,x)}{x^{\frac{1}{2}}} \underset{x \to 0}{\longrightarrow} 0.
\end{equation}

\medskip
\noindent
Straightforward computations lead to $e^{-R\sigma} \PP_{n,\nu,\gamma} g = P_R^+ z + P_R^- z$ where 
\begin{align*}
P_R^+ z :&= (R \sigma_t - R^2 \sigma_x^2) z - z_{xx} + \left( \frac{c_\nu}{x^2} + (n \pi)^2 x^{2 \gamma} \right) z,
\\
P_R^- z :&= z_t - 2 R \sigma_x z_x - R \sigma_{xx} z.
\end{align*}
Then,
\begin{equation} \label{Carleman_lien_z_g}
\iint_{Q_T} P_R^+ z  P_R^- z  \md x \md t \leq \frac{1}{2} \iint_{Q_T} e^{-2 R \sigma} | \PP_{n,\nu,\gamma} g|^2 \md x \md t.
\end{equation}
The rest of the proof follows the classical Carleman strategy \cite{Imanuvilov93} (see \cite{CoronBook} for a pedagogical presentation). We just pay attention to the singular terms.

\medskip
\textit{First step : integrations by part lead to}
\begin{align*}
&\frac{1}{2} \iint_{Q_T} e^{-2 R \sigma} |\PP_{n,\nu,\gamma} g|^2 \md x \md t 
\geq \iint_{Q_T} P_R^+ z  P_R^- z  \md x \md t 
\\
&= R \int_0^T \sigma_x(t,1) z_x^2(t,1) \md t - 2 R \iint_{Q_T} \sigma_{xx} z_x^2 \md x \md t
\\
&+ \iint_{Q_T} \left( -\frac{R}{2} \sigma_{tt} + 2 R^2 \sigma_x \sigma_{xt} - 2 R^3 \sigma_x^2 \sigma_{xx} + \frac{R}{2} \sigma_{xxxx} \right) z^2 \md x \md t
\\
&+ R \iint_{Q_T} \left( -2\frac{c_\nu}{x^3} + 2 \gamma (n \pi)^2  x^{2 \gamma-1} \right) \sigma_x z^2 \md x \md t.
\label{Carleman_IPP}
\tag{\theequation} \addtocounter{equation}{1}
\end{align*}

\noindent
Performing integrations by parts, it is easily seen that $\lag P_R^+z , P_R^-z\rag = I_1 + \dots + I_5$, where
\begin{equation*}
I_1 := \lag (R \sigma_t - R^2 \sigma_x^2) z - z_{xx} , z_t \rag
= \iint_{Q_T}  \Big( -\frac{R}{2}  \sigma_{tt} + R^2 \sigma_x \sigma_{xt} \Big) z^2 \md x \md t,
\end{equation*}

\begin{align*}
I_2 :&= -R^2 \lag \sigma_t z , 2 \sigma_x z_x + \sigma_{xx} z \rag
\\
&= -R^2 \int_0^T \left[ \sigma_t \sigma_x z^2 \right]_0^1 \md t + R^2 \iint_{Q_T} \sigma_{xt} \sigma_x  z^2 \md x \md t,
\end{align*}

\begin{align*}
I_3 :&= R^3 \lag \sigma_x ^2 z , 2 \sigma_x z_x + \sigma_{xx} z \rag
\\
&= R^3 \int_0^T \left[ \sigma_x^3 z^2 \right]_0^1 \md t - R^3 \iint_{Q_T} 2 \sigma_x^2 \sigma_{xx} z^2 \md x \md t,
\end{align*}

\begin{align*}
I_4 :&= R \lag z_{xx} , 2 \sigma_x z_x + \sigma_{xx} z \rag
\\
&= R \int_0^T \left[ \sigma_x z_x^2 \right]_0^1 \md t - R \iint_{Q_T}  \sigma_{xx} z_x^2 \md x \md t 
+ R \int_0^T \left[ \sigma_{xx} z z_x \right]_0^1 \md t 
\\
&- R \iint_{Q_T} \left( \sigma_{xx} z_x^2 + \sigma_{xxx} z z_x \right) \md x \md t
\\
&= R \int_0^T \left( \left[ \sigma_x z_x^2 \right]_0^1 + \left[ \sigma_{xx} z z_x \right]_0^1 
-  \left[ \tfrac{1}{2} \sigma_{xxx} z^2 \right]_0^1 \right) \md t
\\
& + R \iint_{Q_T} \left( \frac{1}{2} \sigma_{xxxx} z^2 - 2\sigma_{xx} z_x^2 \right)  \md x \md t.
\end{align*}
and
\begin{align*}
I_5 :&= \lag \left( \frac{c_\nu}{x^2} + (n \pi)^2 x^{2 \gamma} \right) z , z_t - 2R \sigma_x z_x - R \sigma_{xx} z \rag	
\\
 &= - R \int_0^T \left[ \left( \frac{c_\nu}{x^2} + (n \pi)^2 x^{2 \gamma} \right) \sigma_x z^2 \right]_0^1 \md t 
\\
&+ R \iint_{Q_T} \left( \frac{c_\nu}{x^2} + (n \pi)^2 x^{2 \gamma} \right)_x \sigma_x z^2 \md x \md t.
\end{align*}
Summing these terms and using (\ref{Carleman_lien_z_g}) leads to 
\begin{align*}
&\frac{1}{2} \iint_{Q_T} e^{-2 R \sigma} |\PP_{n,\nu,\gamma} g|^2 \md x \md t 
\geq \iint_{Q_T} P_R^+ z  P_R^- z  \md x \md t 
\\
&= -R^2 \int_0^T \left[ \sigma_t \sigma_x z^2 \right]_0^1 \md t 
+ R^3 \int_0^T \left[ \sigma_x^3 z^2 \right]_0^1 \md t
+ R \int_0^T \left[ \sigma_x z_x^2 \right]_0^1 \md t 
\\
&+ R \int_0^T \left[ \sigma_{xx} z z_x \right]_0^1 \md t
- R \int_0^T \left[ \tfrac{1}{2} \sigma_{xxx} z^2 \right]_0^1 \md t
\\
&- R \int_0^T \left[ \left( \frac{c_\nu}{x^2} + (2 n \pi)^2 |x|^{2 \gamma} \right) \sigma_x z^2 \right]_0^1 \md t
- 2 R \iint_{Q_T} \sigma_{xx} z_x^2 \md x \md t
\\
&+ \iint_{Q_T} \left( -\frac{R}{2} \sigma_{tt} + 2 R^2 \sigma_x \sigma_{xt} - 2 R^3 \sigma_x^2 \sigma_{xx} + \frac{R}{2} \sigma_{xxxx} \right) z^2 \md x \md t
\\
&+ R \iint_{Q_T} \left( -2\frac{c_\nu}{x^3} + 2 \gamma (n \pi)^2  x^{2 \gamma-1} \right) \sigma_x z^2 \md x \md t.
\end{align*}
The weight being regular at $x=1$ and $z(t,1)=0$ every boundary term at $x=1$ vanishes except
\begin{equation*}
R \int_0^T \sigma_x(t,1) z_x^2(t,1) \md t.
\end{equation*}
Using (\ref{limites_z}) and $b>0$ we get that every boundary term at $x=0$ vanish. For example
\begin{equation*}
\sigma_{xxx} z^2 = b (b-1) (b-2)  x^b \, \frac{z^2(x)}{x^3}   \underset{x \to 0}{\longrightarrow} 0.
\end{equation*}
Thus, we get (\ref{Carleman_IPP}).

\bigskip
\textit{Second step : lower bounds on the right-hand side of (\ref{Carleman_IPP}).}
Recall that $\sigma(t,x) = \theta(t) x^b$ with $b$ satisfying (\ref{condition_b}). 

\medskip
\textbf{Boundary term.}
As $b >0$, we have $\sigma_x(t,1) > 0$ and thus
\begin{equation} \label{Carleman_bord}
R \int_0^T \sigma_x(t,1) z_x^2(t,1) \md t \geq 0.
\end{equation}

\smallskip
\textbf{Potential coming from the degeneracy.}
As $\sigma_x(t,x) = b \theta(t) x^{b-1} \geq 0$ on $Q_T$ and $\int_0^1 \frac{z^2}{x^2} \md x < \infty$, it comes that
\begin{equation} \label{Carleman_potentiel_n}
0 \leq R \iint_{Q_T} 2 \gamma (n \pi)^2 x^{2 \gamma-1} \sigma_x z^2 \md x \md t < \infty.
\end{equation}

\smallskip
\textbf{Regular term.}
Let 
\begin{equation*}
I_r := \iint_{Q_T} \left( -\tfrac{R}{2} \sigma_{tt} + 2 R^2 \sigma_x \sigma_{xt} - 2 R^3 \sigma_x^2 \sigma_{xx} \right) z^2 \md x \md t.
\end{equation*}
We prove that for $R$ large enough the leading term in $I_r$ is the one in $R^3$.
From straightforward computations we have
\begin{equation} \label{Carleman_reg1}
-2 R^3 \iint_{Q_T} \sigma_x^2 \sigma_{xx} z^2 \md x \md t = 2 b^3 (1-b) R^3 \iint_{Q_T} \theta^3 x^{3b -4} z^2 \md x \md t.
\end{equation}
Notice that this term is non-negative as $b  \in (0,1)$. Classical computations imply that
\begin{equation*}
|\theta_{tt}(t)| + |\theta(t) \theta_t(t)| \leq C \theta^3(t), \quad \forall t \in (0,T).
\end{equation*}
Here and in the following, $C$ denotes a positive constant that may vary each time it appears. Thus,
\begin{equation*}
\left| \iint_{Q_T} \left(-\tfrac{R}{2} \sigma_{tt} + 2 R^2 \sigma_x \sigma_{xt} \right) z^2 \md x \md t \right|
\leq C \iint_{Q_T} \theta^3 \left( R^2 x^{2 b -2} + R x^b \right) z^2 \md x \md t.
\end{equation*}
As $b \in (0,1)$, for every $x \in (0,1)$, $x^{2b-2} \leq x^{3b-4}$ and $x^b \leq x^{3b-4}$. Hence, as soon as $R \geq 1$,
\begin{equation*}
\left| \iint_{Q_T} \left(-\tfrac{R}{2} \sigma_{tt} + 2 R^2 \sigma_x \sigma_{xt} \right) z^2 \md x \md t \right|
\leq C R^2 \iint_{Q_T} \theta^3 x^{3b-4} z^2 \md x \md t.
\end{equation*}
Together, with (\ref{Carleman_reg1}), we get
\begin{equation*}
I_r \geq C (R^3 -R^2) \iint_{Q_T} \theta^3 x^{3b-4} z^2 \md x \md t.
\end{equation*}
Using, $x^{3b-4} \geq 1$ on $(0,1)$, we get the existence of $C_0$ and $R_0$ positive constants such that for $R \geq R_0$,
\begin{equation} \label{Carleman_reg}
I_r \geq C_0 R^3 \iint_{Q_T} \theta^3 z^2 \md x \md t.
\end{equation}

\smallskip
\textbf{Singular potential.} 
Let 
\begin{equation*}
I_s := -2 R \iint_{Q_T} \sigma_{xx} z_x^2 \md x \md t + R \iint_{Q_T} \left( \frac{1}{2} \sigma_{xxxx} - 2 \frac{c_\nu \sigma_x}{x^3} \right) z^2 \md x \md t.
\end{equation*}
Notice that the two singular potentials are of the same order. Indeed,
\begin{equation*}
\frac{1}{2} \sigma_{xxxx} - 2 \frac{c_\nu \sigma_x}{x^3} =  \frac{b}{2} \left( (b-1) (b-2) (b-3) - 4 c_\nu \right) x^{b-4}.
\end{equation*}
We prove that 
\begin{equation} \label{Carleman_sing}
I_s \geq 0.
\end{equation}
From the Hardy's inequality given in Proposition~\ref{Prop:Hardy}, with $\alpha := b-2$, we get
\begin{align*}
I_s &= 
2R b (1-b) \int_0^T \theta \int_0^1 x^{b-2} z_x^2 \md x \md t 
\\
&+ \frac{b}{2} R \int_0^T \theta \int_0^1 \left( (b-1) (b-2) (b-3) - 4 c_\nu \right) x^{b-4} z^2 \md x \md t
\\
&\geq \frac{b}{2} R \int_0^T \theta \int_0^1 \left( (1-b)(b-3)^2 - (1-b) (b-2) (b-3) - 4 c_\nu \right) x^{b-4} z^2 \md x \md t
\\
&=\frac{b}{2} R \int_0^T \theta \int_0^1 \left( (1-b)(3-b) - 4 c_\nu \right) x^{b-4} z^2 \md x \md t
\end{align*}
Recall that $c_\nu = \nu^2 - \frac{1}{4}$. Thus, if $\nu \in \left( 0, \frac{1}{2} \right]$ we have $c_\nu \leq 0$ and then
\begin{equation*}
(1-b)(3-b) - 4 c_\nu \geq (1-b) (3-b) \geq 0,
\end{equation*}
for any $b \in (0,1)$. This gives (\ref{Carleman_sing}).

If $\nu \in \left( \frac{1}{2}, 1 \right)$, setting $b = 2- 2 \nu$ we still have $b \in (0,1)$ and 
\begin{equation*}
(1-b)(3-b) - 4 c_\nu = 0.
\end{equation*}
This gives (\ref{Carleman_sing}).

\smallskip
\textbf{Conclusion.} Gathering (\ref{Carleman_bord}), (\ref{Carleman_potentiel_n}), (\ref{Carleman_reg}) and (\ref{Carleman_sing}) in (\ref{Carleman_IPP}) we get that for $R \geq R_0$,
\begin{equation*}
\frac{1}{2} \iint_{Q_T} e^{-2 R \sigma} |\PP_{n,\nu,\gamma} g|^2 \md x \md t  
\geq C_0 R^3 \iint_{Q_T} \theta^3 z^2 \md x \md t.
\end{equation*}
This ends the proof of Proposition~\ref{prop_Carleman}.

\end{proof}

\begin{rmq} \label{rk_poids_Carleman2}
We here point out some of the differences between Proposition~\ref{prop_Carleman} and the Carleman estimates established in the case of a boundary inverse square singularity in \cite{VancostenobleZuazua, Vancostenoble11}. 
In both estimates the singular potential appears as 
\begin{equation*}
\iint_{Q_T} \frac{\sigma_x}{x^3} z^2 \md x \md t.
\end{equation*}
In \cite{VancostenobleZuazua}, the weight is defined by $p(x) = 1- \frac{x^2}{2}$. Thus, the singular potential can be treated with some classical Hardy type inequalities. 

In our situation, using the extra information $z_x(t,0)=0$, we are able to deal with a weight function with singular derivatives. The weight is chosen concave so that the term in $R^3$ is the leading one. The weight is chosen increasing to deal with the boundary term $\sigma_x(t,1) z_x^2(t,1)$. At the same time, this allows to deal with the potential coming from the degeneracy (\ref{Carleman_potentiel_n}). 
The price to pay is that we have to handle very singular terms of the form
\begin{equation*}
\iint_{Q_T} \theta x^b \frac{z^2}{x^4} \md x \md t.
\end{equation*}
Thus the Carleman estimate stated in Proposition~\ref{prop_Carleman} only holds because of the extra information $z_x(t,0)=0$.
\end{rmq}
\medskip

\begin{rmq} \label{rk_conditions_periodiques}
To be closer to the setting studied by Boscain and Laurent \cite{BoscainLaurent}, we could study for $(t,x,y) \in (0,T) \times (-1,1) \times \T$,
\begin{equation} \label{grushin_2d_tore}
\left\{
\begin{aligned}
& \partial_t f - \partial^2_{xx} f - |x|^{2 \gamma} \partial^2_{yy} f + \frac{\gamma}{2} \left( \frac{\gamma}{2} + 1 \right) \frac{1}{x^2} f = u(t,x,y) \chi_{\omega}(x,y),   
\\
&f(t,-1,y) = f(t,1,y) =  0,
\\
&f(0,x,y) = f^0(x,y).
\end{aligned}
\right.
\end{equation}
Defining the semigroup as 
\begin{equation*}
T(t) f^0(x,y) := \sum_{n \in \Z} f_n(t,x) e^{i n y},
\end{equation*}
with $f_n := e^{-t A_{\frac{n}{\pi}} } f_n^0$ we get that its infinitesimal generator is an extension of the singular Grushin operator on $C^\infty_0([ (-1,0) \cup (0,1) ] \times \T)$. As in Proposition~\ref{prop_bien_pose_2d}, this semigroup leads to a unique mild solution of (\ref{grushin_2d_tore}). 
As $0 < \frac{\gamma}{2} \left( \frac{\gamma}{2} + 1 \right) < \frac{3}{4}$ for $\gamma \in (0,1)$, essentially with the same proof as Sect.~\ref{sect_continuation_unique}, we would obtain approximate controllability for any $\gamma \in (0,1)$.
\end{rmq}

\section{Conclusion, open problems and perspectives}

In this paper we have investigated the approximate controllability properties for a \textsc{2d} Grushin equation which presents both a degeneracy and an inverse square singularity on the internal set $\{x=0\}$. As the associated operator possesses several self-adjoint extensions, the functional setting in which we study the well posedness and unique continuation for the adjoint system is crucial. This functional setting relies on a precise study of the \textsc{1d} associated operators and the design of a self-adjoint extension of the singular operator with suitable transmission conditions across the singularity.

Using classical unique continuation results for uniformly parabolic operators, the study of unique continuation is reduced to the study of a \textsc{1d} problem with a boundary inverse square singularity. The proof of unique continuation is ended with a suitable Carleman type estimate that relies on a Hardy inequality.

An interesting open problem coming from this work is the question of null controllability in the case $\nu \in (0,1)$. The classical strategy would be to prove uniform observability for the \textsc{1d} adjoint systems. This has been done in the case where there is no singular potential in \cite{BeauchardCannarsaGuglielmi} and with a boundary singular potential in \cite{CannarsaGuglielmi13}. The Carleman type estimate we proved in this paper might not be directly used as it holds true only for the regular part of the coefficient $g_n$. Dealing with the singular part in Carleman type estimates is quite tricky as we cannot perform integrations by part on the singular part. The other difficulty relies on the fact that we want these estimates to be uniform with respect to $n$.

\paragraph*{Acknowledgements :} The author thanks K.~Beauchard for having drawn his attention to this problem and for fruitful discussions. The author thanks M.~Gueye and D.~Prandi for interesting discussions.

\bibliography{biblio}

\begin{thebibliography}{10}

\bibitem{AlekseevaAnanieva}
V.S. Alekseeva and A.Y. Ananieva.
\newblock On extensions of the {B}essel operator on a finite interval and a
  half line.
\newblock {\em Journal of Mathematical Sciences}, 187, 2012.

\bibitem{BarasGoldstein84}
P.~Baras and J.A. Goldstein.
\newblock The heat equation with a singular potential.
\newblock {\em Trans. Amer. Math. Soc.}, 284(1):121--139, 1984.

\bibitem{BeauchardCannarsaGuglielmi}
K.~Beauchard, P.~Cannarsa, and R.~Guglielmi.
\newblock Null controllability of {G}rushin-type operators in dimension two.
\newblock {\em J. Eur. Math. Soc. (JEMS)}, 16(1):67--101, 2014.

\bibitem{BoscainLaurent}
U.~Boscain and C.~Laurent.
\newblock The {L}aplace-{B}eltrami operator in almost-{R}iemannian {G}eometry.
\newblock To appear in \textit{Ann. Inst. Fourier}, preprint arXiv:1105.4687,
  2011.

\bibitem{BoscainPrandi13}
U.~Boscain and D.~Prandi.
\newblock The laplace-beltrami operator on conic and anticonic-type surfaces.
\newblock preprint, arXiv:1305.5271, 2013.

\bibitem{CabreMartel99}
X.~Cabr{\'e} and Y.~Martel.
\newblock Existence versus explosion instantan\'ee pour des \'equations de la
  chaleur lin\'eaires avec potentiel singulier.
\newblock {\em C. R. Acad. Sci. Paris S\'er. I Math.}, 329(11):973--978, 1999.

\bibitem{CannarsaGuglielmi13}
P.~Cannarsa and R.~Guglielmi.
\newblock Null controllability in large time for the parabolic {G}rushin
  operator with singular potential.
\newblock preprint, 2013.

\bibitem{CannarsaMartinezVancostenoble05}
P.~Cannarsa, P.~Martinez, and J.~Vancostenoble.
\newblock Null controllability of degenerate heat equations.
\newblock {\em Adv. Differential Equations}, 10(2):153--190, 2005.

\bibitem{CannarsaMartinezVancostenoble08}
P.~Cannarsa, P.~Martinez, and J.~Vancostenoble.
\newblock Carleman estimates for a class of degenerate parabolic operators.
\newblock {\em SIAM J. Control Optim.}, 47(1):1--19, 2008.

\bibitem{CannarsaMartinezVancostenoble09}
P.~Cannarsa, P.~Martinez, and J.~Vancostenoble.
\newblock Carleman estimates and null controllability for boundary-degenerate
  parabolic operators.
\newblock {\em C. R. Math. Acad. Sci. Paris}, 347(3-4):147--152, 2009.

\bibitem{CannarsaTortYamamoto12}
P.~Cannarsa, J.~Tort, and M.~Yamamoto.
\newblock Unique continuation and approximate controllability for a degenerate
  parabolic equation.
\newblock {\em Appl. Anal.}, 91(8):1409--1425, 2012.

\bibitem{CazenaveHarauxBook}
T.~Cazenave and A.~Haraux.
\newblock {\em Introduction aux probl\`emes d'\'evolution semi-lin\'eaires},
  volume~1 of {\em Math\'ematiques \& Applications (Paris) [Mathematics and
  Applications]}.
\newblock Ellipses, Paris, 1990.

\bibitem{CoronBook}
J.-M. Coron.
\newblock {\em Control and nonlinearity}, volume 136 of {\em Mathematical
  Surveys and Monographs}.
\newblock American Mathematical Society, Providence, RI, 2007.

\bibitem{Ervedoza08}
S.~Ervedoza.
\newblock Control and stabilization properties for a singular heat equation
  with an inverse-square potential.
\newblock {\em Comm. Partial Differential Equations}, 33(10-12):1996--2019,
  2008.

\bibitem{Gueye13}
M.~Gueye.
\newblock Exact boundary controllability of {1-D} parabolic and hyperbolic
  degenerate equations.
\newblock preprint, 2013.

\bibitem{Imanuvilov93}
O.Y. Imanuvilov.
\newblock Boundary controllability of parabolic equations.
\newblock {\em Uspekhi Mat. Nauk}, 48(3(291)):211--212, 1993.

\bibitem{MartinezVancostenoble06}
P.~Martinez and J.~Vancostenoble.
\newblock Carleman estimates for one-dimensional degenerate heat equations.
\newblock {\em J. Evol. Equ.}, 6(2):325--362, 2006.

\bibitem{PazyBook}
A.~Pazy.
\newblock {\em Semigroups of linear operators and applications to partial
  differential equations}, volume~44 of {\em Applied Mathematical Sciences}.
\newblock Springer-Verlag, New York, 1983.

\bibitem{ReedSimon2}
M.~Reed and B.~Simon.
\newblock {\em Methods of modern mathematical physics. {II}. {F}ourier
  analysis, self-adjointness}.
\newblock Academic Press [Harcourt Brace Jovanovich Publishers], New York,
  1975.

\bibitem{SautScheurer87}
J.-C. Saut and B.~Scheurer.
\newblock Unique continuation for some evolution equations.
\newblock {\em J. Differential Equations}, 66(1):118--139, 1987.

\bibitem{Vancostenoble11}
J.~Vancostenoble.
\newblock Improved {H}ardy-{P}oincar\'e inequalities and sharp {C}arleman
  estimates for degenerate/singular parabolic problems.
\newblock {\em Discrete Contin. Dyn. Syst. Ser. S}, 4(3):761--790, 2011.

\bibitem{VancostenobleZuazua}
J.~Vancostenoble and E.~Zuazua.
\newblock Null controllability for the heat equation with singular
  inverse-square potentials.
\newblock {\em J. Funct. Anal.}, 254(7):1864--1902, 2008.

\bibitem{VazquezZuazua00}
J.L. Vazquez and E.~Zuazua.
\newblock The {H}ardy inequality and the asymptotic behaviour of the heat
  equation with an inverse-square potential.
\newblock {\em J. Funct. Anal.}, 173(1):103--153, 2000.

\bibitem{ZettlBook}
A.~Zettl.
\newblock {\em Sturm-{L}iouville theory}, volume 121 of {\em Mathematical
  Surveys and Monographs}.
\newblock American Mathematical Society, Providence, RI, 2005.

\end{thebibliography}
\bibliographystyle{plain}

\end{document}